\documentclass{siamltex1213}
\usepackage[utf8]{inputenc}
\usepackage{amsmath,amsfonts,graphicx}
\title{NUMERICAL SOLUTION OF THE TIME-FRACTIONAL FOKKER--PLANCK 
EQUATION WITH GENERAL FORCING\thanks{This work was supported by the 
Australian Research 
Council grant DP140101193.}}
\author{Kim Ngan Le\footnotemark[2]
\and William McLean\footnotemark[2] 
\and Kassem Mustapha\footnotemark[3]}
\newcommand{\Sh}{\mathbb{S}_h}
\newcommand{\R}{\mathbb{R}}
\newcommand{\iprod}[1]{\langle#1\rangle}
\newcommand{\bigiprod}[1]{\bigl\langle#1\bigr\rangle}
\newcommand{\gtilde}{g_*}
\newcommand{\vecu}{\boldsymbol{u}}
\newcommand{\vecv}{\boldsymbol{v}}
\newcommand{\vecg}{\boldsymbol{g}}
\newcommand{\matk}{\boldsymbol{K}}
\newcommand{\atilde}{\tilde a}
\newcommand{\Utilde}{\widetilde U}
\newcommand{\Vtilde}{\widetilde V}
\newcommand{\matM}{\boldsymbol{M}}
\newcommand{\matB}{\boldsymbol{B}}
\newcommand{\vecU}{\boldsymbol{U}}
\newcommand{\vecG}{\boldsymbol{G}}
\begin{document}

\maketitle

\renewcommand{\thefootnote}{\fnsymbol{footnote}}
\footnotetext[2]{School of Mathematics and Statistics, The University 
of New South Wales, Sydney 2052, Australia.}
\footnotetext[3]{Department of Mathematics and Statistics, King Fahd 
University of Petroleum and Minerals, Dhahran 31261, Saudi Arabia.}
\renewcommand{\thefootnote}{\arabic{footnote}}
\begin{abstract}
We study two schemes for a time-fractional Fokker--Planck equation 
with space- and time-dependent forcing in one space dimension.  The 
first scheme is continuous in time and is discretized in space using 
a piecewise-linear Galerkin finite element method.  The second is 
continuous in space and employs a time-stepping 
procedure similar to the classical implicit Euler method.  We show 
that the space discretization is second-order accurate in the 
spatial $L_2$-norm, uniformly in time, whereas the 
corresponding error for the time-stepping scheme is $O(k^\alpha)$ for 
a uniform time step~$k$, where $\alpha\in(1/2,1)$ is the fractional 
diffusion parameter.  In numerical experiments using a combined, 
fully-discrete method, we observe convergence behaviour consistent 
with these results.
\end{abstract}

\section{Introduction}
We investigate the numerical solution of the inhomogeneous, 
time-fractional Fokker--Planck 
equation~\cite{HenryLanglandsStraka2010},
\begin{equation}\label{eq: ibvp}
u_t-\kappa_\alpha\partial_t^{1-\alpha}u_{xx}
	+\mu_\alpha^{-1}\bigl(F\partial_t^{1-\alpha}u\bigr)_x
	=g,
\end{equation}
for $0<x<L$ and $0<t<T$, with initial data~$u(x,0)=u_0(x)$ and 
subject to homogeneous Dirichlet boundary conditions
$u(0,t)=0=u(L,t)$.  (We use subscripts to indicate partial 
derivatives of integer order with respect to $x$~or $t$; for instance,
$u_t=\partial u/\partial t$.)
The parameter~$\kappa_\alpha$ is the  generalized diffusivity 
constant,  $\mu_\alpha$ is the  generalized friction constant, and the 
driving force~$F$ and the source term~$g$ are permitted to be 
functions of both $x$~and $t$.  The subdiffusion parameter~$\alpha$ 
satisfies $0<\alpha<1$, and the fractional time derivative is 
interpreted in the Riemann--Liouville sense; thus, 
$\partial_t^{1-\alpha}=(I^\alpha v)_t$ where $I^\alpha$ is the 
fractional integral of order~$\alpha$,
\begin{align*}
I^\alpha v(t)=\int_0^t \omega_\alpha(t-s) v(s)\,ds
\quad\text{with}\quad
\omega_{\alpha}(t)=\frac{t^{\alpha-1}}{\Gamma(\alpha)}.
\end{align*}

In~1999, Metzler et al.~\cite{MetzlerBarkaiKlafter1999} used a 
discrete master equation to model the behaviour of subdiffusive 
particles in the presence of a driving force~$F(x)$, showing that 
in the diffusive limit the probability density~$u(x,t)$ 
for a particle to be at position~$x$ at time~$t$ obeys a fractional 
Fokker--Planck equation of the form
\begin{equation}\label{eq: ibvp 2}
\begin{aligned}
u_t-\partial_t^{1-\alpha}\bigl(\kappa_\alpha u_{xx}
	-\mu_\alpha^{-1}\bigl(F u\bigr)_x\bigr)=0.
\end{aligned}
\end{equation}
Subsequently, Henry et al.~\cite{HenryLanglandsStraka2010} considered 
the more general case when~$F=F(x,t)$ may depend on~$t$ as well 
as~$x$, and showed that $u$ obeys \eqref{eq: ibvp} with~$g\equiv0$.
The two equations coincide if $F$ is independent of~$t$, but if the 
forcing is time-dependent then \eqref{eq: ibvp 2}
does not properly correspond to any physical stochastic 
process~\cite{HeinsaluPatriarcaGoychukHanggi2007}.

Various numerical (time stepping finite difference) methods have been
proposed for solving \eqref{eq: ibvp 2}, usually for $F$ 
assumed to be either a constant or a function of~$x$ only. The 
starting point was often to rewrite equation~\eqref{eq: ibvp 2} in 
the form
\begin{equation}\label{eq: ibvp 3}
I^{1-\alpha}(u_t)-\kappa_\alpha u_{xx}+\mu_\alpha^{-1}(F u)_x=0,
\end{equation}
in which the first term is a Caputo fractional derivative.
Indeed, \eqref{eq: ibvp 3} is in some ways more convenient than 
\eqref{eq: ibvp 2} for constructing and analyzing the accuracy of 
numerical schemes.  However, the simpler form~\eqref{eq: ibvp 3} is 
not applicable in our case because $F$ may depend on~$t$.

For the numerical solution of~\eqref{eq: ibvp 3} with~$F=F(x)$, 
Deng~\cite{Deng2007} transformed the equation to a system of 
fractional ODEs by discretizing the spatial derivatives and using the 
properties of Riemann--Liouville and Caputo fractional derivatives, 
and then applied a predictor-corrector approach combined with the 
method of lines.  This work also presented a stability and 
convergence analysis. Cao et al.~\cite{CaoFuHuang2012} adopted a 
similar approach for~\eqref{eq: ibvp 2} and solved the resulting 
system of fractional ODEs using a second order, backward Euler scheme. 
Chen et al.~\cite{ChenLiuZhuangAnh2009} studied the stability 
and convergence properties of three implicit finite difference 
techniques, in each of which the diffusion term was approximated by 
the standard second order difference approximation at 
the advanced time level. In related work, Jiang~\cite{Jiang2015}
established monotonicity properties of the numerical solutions 
obtained by using these schemes, and so showed that the time-stepping 
preserves non-negativity of the solution. Based on this property, a 
new proof of stability and convergence was provided. 

Fairweather et al.~\cite{FairweatherZhangYangXu2014} investigated the 
stability and convergence of an orthogonal spline collocation method  
in space combined  with the backward Euler method in time, based on 
the L1~approximation of the fractional derivative. In an earlier 
work,  Saadmandi~\cite{SaadatmandiDehghanAzizi2012} studied a 
collocation method based on shifted Legendre polynomials in time and 
Sinc functions in space. Recently, Vong and Wang~\cite{VongWang2015} 
have analysed a high order, compact difference scheme 
for~\eqref{eq: ibvp 3}, and Cui~\cite{Cui2015} has considered a 
more general fractional convection--diffusion equation, 
\[
I^{1-\alpha}u_t-(au_x)_x+bu_x+cu=f, 
\]
with coefficients $a$, $b$, $c$ that may depend on $x$~and $t$, 
applying a high-order approximation for the time fractional derivative 
combined with a compact exponential finite difference scheme for 
approximating the convection and diffusion terms.  Stability (using 
Fourier methods) and an estimate for the local truncation error were  
obtained in the case of constant coefficients.  We are not aware of 
any previous analysis on the numerical solution of~\eqref{eq: ibvp} 
for a general~$F$ depending on both $x$~and $t$.  
  
In Section~\ref{sec: prelim} we gather together some preliminary 
results needed in our subsequent analysis, including continuous and 
discrete versions of a generalized Gronwall inequality involving the 
Mittag--Leffler function in place of the usual exponential.  One of 
these results (Lemma~\ref{lem: partial v v_t}) holds only 
for~$1/2<\alpha<1$ so much of our theory requires this restriction. 
Section~\ref{sec: space} deals with a spatial discretization 
of~\eqref{eq: ibvp} by a continuous, piecewise-linear Galerkin finite 
element method.  We prove stability of the scheme in Theorem~\ref{thm: 
stability Fx} and, under weaker assumptions on~$F$ but with a worse 
bound, in Theorem~\ref{thm: stability E}.  An error estimate follows 
in Theorem~\ref{thm: uh error} showing second-order accuracy 
in~$L_\infty\bigl((0,T),L_2(\Omega)\bigr)$, where $\Omega=(0,L)$ 
denotes the spatial interval.  We 
then study a time stepping scheme in Section~\ref{sec: time}, proving 
a stability estimate in Theorem~\ref{thm: time stepping stability} 
and then an error bound in Theorem~\ref{thm: time error}, assuming a 
constant time step~$k$.  This scheme, which is continuous in space, is 
formally first-order accurate but, owing to the weakly singular kernel 
in the fractional integral, we are able to show only that the error 
in $L_2(\Omega)$ at the $n$th time level is $O(k^\alpha)$.  
Section~\ref{sec: numerical} reports on numerical experiments with 
a fully discrete scheme based on the semi-discrete ones analyzed in 
Sections \ref{sec: space}~and \ref{sec: time}.  We observe $O(k+h^2)$ 
convergence when $\alpha$ is close to~$1$, or when we use an 
appropriately graded mesh in time.  The experiments give no evidence 
that the methods fail if~$0<\alpha\le1/2$, although the convergence 
rate deteriorates as~$\alpha$ decreases when using a uniform time 
step.  We also apply our method to a problem from a recent paper by 
Angstmann et al.~\cite{AngstmannEtAl2015} and investigate whether the 
regularity of the initial data affects the stability of the methods. 
A brief appendix proves a technical result 
(Lemma~\ref{lem: a_n positive}) used in showing stability of the 
time-stepping procedure.
\section{Technical preliminaries}\label{sec: prelim}

Lemmas~\ref{lem: ||v(t)||^2}--\ref{lem: partial v v_t} below 
summarize some properties of fractional integrals that will be needed 
in our analysis.  In each case, we assume that the function $v(t)$ is 
defined for~$0\le t\le T$ and takes values in a Hilbert space with 
inner product~$\iprod{\cdot,\cdot}$ and norm~$\|\cdot\|$, and is 
sufficiently regular for the integrand on the right-hand side to be 
absolutely integrable.

\begin{lemma}\label{lem: ||v(t)||^2}
For $0<\alpha<1$ and $0\le t\le T$,
\[
\|v(t)-v(0)\|^2\le\frac{t^{1-\alpha}}{1-\alpha}
	\int_0^t\|I^{\alpha/2}v_t(s)\|^2\,ds.
\]
\end{lemma}
\begin{proof}
Put $w(t)=I^{\alpha/2}v_t$ so that 
$v(t)-v(0)=I^1v_t=I^{1-\alpha/2}w(t)$ and
\begin{align*}
\|v(t)-v(0)\|^2&\le\biggl(\int_0^t\omega_{1-\alpha/2}(t-s)\|w(s)\|\,ds
	\biggr)^2\\
	&\le\biggl(\int_0^t\frac{(t-s)^{-\alpha}}{\Gamma(1-\alpha/2)^2}
	\,ds\biggr)\biggl(\int_0^t\|w(s)\|^2\,ds\biggr),
\end{align*}
giving the desired bound, because $\Gamma(1-\alpha/2)\ge1$.
\end{proof}

\begin{lemma}\label{lem: I alpha/2} 
For $0<\alpha<1$,
\[
\int_0^T\|I^{\alpha/2}v(t)\|^2\,dt  
	\le\frac{1}{\cos(\alpha\pi/2)}
	\int_0^T\bigiprod{I^\alpha v(t),v(t)}\,dt.
\]
\end{lemma}
\begin{proof}
Mustapha and Sch\"otzau~\cite[Lemma~3.1 (ii)]{MustaphaSchoetzau2014}. 
\end{proof}

\begin{lemma}\label{lem: int I alpha}
For $0<\alpha<1$,
\[
\int_0^T\|I^\alpha v(t)\|^2\,dt\le\omega_{\alpha+1}(T)
	\int_0^T\omega_\alpha(T-t)\int_0^t\|v(s)\|^2\,ds\,dt.
\]
\end{lemma}
\begin{proof}
Since
\begin{align*}
\|I^\alpha v(t)\|^2&\le\biggl(
	\int_0^t\omega_\alpha(t-s)\|v(s)\|\,ds\biggr)^2\\
	&\le\biggl(\int_0^t\omega_\alpha(t-s)\,ds\biggr)
	\biggl(\int_0^t\omega_\alpha(t-s)\|v(s)\|^2\,ds\biggr)\\
	&=\omega_{\alpha+1}(t)\int_0^t
		\omega_\alpha(t-s)\|v(s)\|^2\,ds
\end{align*}
we have
\[
\int_0^T\|I^\alpha v(t)\|^2\,dt\le\omega_{\alpha+1}(T)
	\int_0^T\int_0^t\omega_\alpha(t-s)\|v(s)\|^2\,ds\,dt,
\]
and the double integral on the right equals
\[
\int_0^T\|v(s)\|^2\int_s^T\omega_\alpha(t-s)\,dt\,ds
=\int_0^T\|v(s)\|^2\int_s^T\omega_\alpha(T-t)\,dt\,ds.
\]
The result follows after reversing the order of integration again.
\end{proof}

\begin{lemma}\label{lem: partial v v_t}
For $1/2<\alpha<1$, 
\[
\int_0^T\|\partial_t^{1-\alpha}v(t)\|^2\,dt
		\le\frac{1}{(2\alpha-1)\Gamma(\alpha)^2}\biggl(
	T^{2\alpha-1}\|v(0)\|^2+T^{2\alpha}\int_0^T\|v_t\|^2\,dt\biggr).
\]
\end{lemma}
\begin{proof}
The identity 
$\partial_t^{1-\alpha}v(t)=v(0)\omega_\alpha(t)+I^\alpha v_t(t)$
implies that
\[
\int_0^T\|\partial_t^{1-\alpha}v(t)\|^2\,dt
	\le2\|v(0)\|^2\int_0^T\omega_\alpha(t)^2\,dt
	+2\int_0^T\|I^\alpha v_t\|^2\,dt,
\]
and the Cauchy--Schwarz inequality gives
\begin{align*}
\int_0^T\|I^\alpha v_t\|^2\,dt
	&\le\int_0^T\biggl(\int_0^t\omega_\alpha(t-s)\|v_t(s)\|\,ds
	\biggr)^2\,dt\\
	&\le\int_0^T\biggl(\int_0^t\omega_\alpha(t-s)^2\,ds\biggr)
	\biggl(\int_0^T\|v_t(s)\|^2\,ds\biggr)\,dt,
\end{align*}
so it suffices to note that
$\int_0^t\omega_\alpha(t-s)^2\,ds
\le T^{2\alpha-1}/\bigl((2\alpha-1)\Gamma(\alpha)^2\bigr)$
for~$\alpha>1/2$.
\end{proof}

The existence and uniqueness of our spatially discrete solution 
to~\eqref{eq: ibvp} will follow from the following result for
an $m\times m$~system of weakly singular integral equations.  Here,
$|\cdot|$ may denote any matrix norm on~$\R^{m\times m}$ induced by a 
norm on~$\R^m$.

\begin{theorem}\label{thm: Volterra}
There exists a unique continuous solution~$\vecu:[0,\infty)\to\R^m$
to the linear Volterra integral equation
\[
\vecu(t)+\int_0^t\matk(t,s)\vecu(s)\,ds=\vecg(t)\quad
	\text{for $0\le t<\infty$,}
\]
if the following conditions are satisfied:
\begin{enumerate}
\item $\vecg:[0,\infty)\to\R^m$ is continuous;
\item $\matk(t,s)\in\R^{m\times m}$ is continuous 
for~$0\le s<t<\infty$;
\item for any continuous function~$\vecv:[0,\infty)\to\R^m$, the 
integrals
\[
\int_0^t\matk(t,s)\vecv(s)\,ds
\quad\text{and}\quad
\int_0^t|\matk(t,s)|\,ds
\]
exist and are continuous for~$0\le t<\infty$;
\item there exist constants $\gamma>0$~and $\epsilon>0$ such that
\[
\int_0^t e^{-\gamma(t-s)}|\matk(t,s)|\,ds\le1-\epsilon
	\quad\text{for $0\le t<\infty$.}
\]
\end{enumerate}
\end{theorem}
\begin{proof}
Becker~\cite[Corollary~2.3]{Becker2011}.
\end{proof}

Our stability analysis of the spatially discrete solution makes use of 
the following weakly singular Gronwall inequality, involving the 
Mittag--Leffler function
\begin{equation}\label{eq: ML func}
E_\beta(z)=\sum_{n=0}^\infty\frac{z^n}{\Gamma(1+n\beta)}.
\end{equation}
The usual Gronwall inequality is just the special case~$\beta=1$,
because $E_1(z)=e^z$.

\begin{lemma}\label{lem: Gronwall}
Let $\beta>0$ and $T>0$.  Assume that $a$~and $b$ are non-negative 
and non-decreasing functions on the interval~$[0,T]$. If
$y:[0,T]\to\R$ is a locally integrable function satisfying
\[
0\le y(t)\le a(t)+b(t)\int_0^t\omega_\beta(t-s)y(s)\,ds
	\quad\text{for $0\le t\le T$,}
\]
then
\[
y(t)\le a(t)E_\beta\bigl(b(t)t^\beta\bigr)
\quad\text{for $0\le t\le T$.}
\]
\end{lemma}
\begin{proof}
Dixon and McKee~\cite[Theorem~3.1]{DixonMcKee1986};
Ye, Gao and Ding~\cite[Corollary~2]{YeGaoDing2007}.
\end{proof}

We also use a discrete version of this Gronwall inequality to 
establish stability of our time stepping procedure.

\begin{lemma}\label{lem: discrete Gronwall}
Let $0<\beta\le1$, $N>0$, $k>0$ and $t_n=nk$ for $0\le n\le N$.  
Assume that $(A_n)_{n=0}^N$ is a non-negative and non-decreasing 
sequence, and that $B\ge0$. If the sequence~$(y^n)_{n=0}^N$ satisfies
\[
0\le y^n\le A_n+Bk\sum_{j=0}^{n-1}\omega_\beta(t_n-t_j)y^j 
	\quad\text{for $0\le n\le N$,}
\]
then 
\[
y^n\le A_nE_\beta(Bt_n^\beta) \quad\text{for $0\le n\le N$.}
\]
\end{lemma}
\begin{proof}
Dixon and McKee~\cite[Theorem~6.1]{DixonMcKee1986}.
\end{proof}

\section{Spatial discretization}\label{sec: space}
We choose a partition $0=x_0<x_1<x_2<\cdots<x_P=L$ of the spatial 
interval~$\Omega=(0,L)$ and denote the length of the $p$th 
subinterval by~$h_p=x_p-x_{p-1}$ for~$1\le p\le P$.  
With~$h=\max_{1\le p\le P}h_p$, we define the usual space~$\Sh$ of 
continuous, piecewise-linear functions that satisfy the Dirichlet 
boundary conditions, so that $\Sh\subseteq H^1_0(\Omega)$.  Recall
that $F=F(x,t)$~and $g=g(x,t)$.  In our notation, we will often 
suppress the dependence on~$x$ and think of $u=u(x,t)$ as a function 
of~$t$ taking values in~$L_2(\Omega)$.  We also assume that 
$\kappa_\alpha=\mu_\alpha=1$.

In the usual weak formulation of~\eqref{eq: ibvp}, we seek $u$ 
satisfying 
\begin{equation}\label{eq: ibvp weak}
\iprod{u_t,v}+\iprod{\partial_t^{1-\alpha}u_x,v_x}
	-\iprod{F\partial_t^{1-\alpha}u,v_x}=\iprod{g(t),v}
	\quad\text{for $v\in H^1_0(\Omega)$,}
\end{equation}
with $u(0)=u_0$, where $\iprod{\cdot,\cdot}$ denotes the inner 
product in~$L_2(\Omega)$.  For our error analysis, it is useful to 
consider a slightly more general, spatially discrete version 
of~\eqref{eq: ibvp weak}, in which $u_h:[0,T]\to\Sh$ satisfies
\begin{equation}\label{eq: spatially discrete}
\iprod{u_{ht},v}+\iprod{\partial_t^{1-\alpha}u_{hx},v_x}
	-\iprod{F\partial_t^{1-\alpha}u_h,v_x} 
	=\iprod{g(t),v}+\iprod{\gtilde(t),v_x}\\
\end{equation}
for $v\in\Sh$, with $u_h(0)=u_{0h}$ where~$u_0\approx u_{0h}\in\Sh$,
and where $u_{ht}=\partial u_h/\partial t$.  Thus, if 
$\gtilde(x,t)\equiv0$, then $u_h$ is the standard Galerkin finite 
element solution of~\eqref{eq: ibvp weak}.

To show the existence and uniqueness of~$u_h$ 
satisfying~\eqref{eq: spatially discrete}, define the linear 
operator~$B_h(t):\Sh\to\Sh$ (which depends on~$t$ through~$F$) by
\[
\iprod{B_h(t)v,w}=\iprod{v_x,w_x}-\iprod{Fv,w_x}
	\quad\text{for $v$, $w\in\Sh$,}
\]
and the finite element function~$g_h(t)\in\Sh$ by
\[
\iprod{g_h(t),w}=\iprod{g(t),w}+\iprod{\gtilde(t),w_x}
	\quad\text{for $w\in\Sh$.}
\]
The variational equation~\eqref{eq: spatially discrete} is 
then equivalent to
\[
u_{ht}+B_h(t)\partial_t^{1-\alpha}u_h=g_h(t).
\]
Integrating with respect to~$t$, we find that $u_h$ satisfies the 
Volterra equation
\[
u_h(t)+\int_0^t K_h(t,s)u_h(s)\,ds=G_h(t)\quad
	\text{for $0\le t\le T$,}
\]
with the weakly-singular kernel
\[
K_h(t,s)=B_h(t)\omega_\alpha(t-s)
		-\int_s^tB_{ht}(\tau)\omega_\alpha(\tau-s)\,d\tau
\]
and right-hand side
\[
G_h(t)=u_{0h}+\int_0^t g_h(s)\,ds.
\]

\begin{theorem}\label{thm: uh existence}
If $F\in W^1_\infty\bigl((0,T);L_\infty(\Omega)\bigr)$ and 
$g$, $\gtilde\in L_1\bigl((0,T);L_2(\Omega)\bigr)$, then for 
any~$u_{0h}\in\Sh$ there exists a unique continuous 
$u_h:[0,\infty)\to\Sh$ satisfying \eqref{eq: spatially discrete} for 
all~$v\in\Sh$, with $u_h(0)=u_{0h}$.
\end{theorem}
\begin{proof}
Let $|\cdot|$ denote any norm on the finite dimensional 
space~$\Sh$.  Our assumptions on $F$, $g$ and $\gtilde$ ensure that 
$G_h$ satisfies condition~1 of Theorem~\ref{thm: Volterra}, and that 
$K_h$ satisfies conditions 2~and 3 (after fixing any basis 
for~$\Sh$).  Furthermore,
\[
|K_h(t,s)|\le C_{F,h}\biggl(\omega_\alpha(t-s)
	+\int_s^t\omega_\alpha(\tau-s)\,d\tau\biggr)
	=C_{F,h}\bigl[\omega_\alpha(t-s)+\omega_{1+\alpha}(t-s)\bigr],
\]
and, denoting the Laplace transform by~$\mathcal{L}$,
\[
\int_0^te^{-\gamma(t-s)}\omega_\alpha(t-s)\,ds
	\le\int_0^\infty e^{-\gamma s}\omega_\alpha(s)\,ds
	=\mathcal{L}\omega_\alpha(\gamma)=\gamma^{-\alpha},
\]
so
\[
\int_0^te^{-\gamma(t-s)}|K_h(t,s)|\,ds
	\le C_{F,h}[\gamma^{-\alpha}+\gamma^{-1-\alpha}],
\]
and condition~4 follows for~$\gamma$ sufficiently large.
\end{proof}

Theorem~\ref{thm: uh existence} gives no meaningful stability result 
for~$u_h(t)$ (because $C_{F,h}$ from the proof grows rapidly
as~$h\to0$) but an energy argument yields the following estimate.  We 
use the abbreviation $\|v\|_r$ for the norm in~$H^r(\Omega)$.

\begin{theorem}\label{thm: stability Fx}
If, in addition to the assumptions of Theorem~\ref{thm: uh existence},
\begin{enumerate}
\item $F_x(x,t)\ge0$ for $0<x<L$~and $0<t<T$;
\item $1+F(x,t)^2\le C_F$ for $0<x<L$ and $0<t<T$;
\item $1/2<\alpha<1$;
\end{enumerate}
then
\begin{multline*}
\|u_h(t)-u_{0h}\|^2\le\frac{t^{1-\alpha}}{(1-\alpha)^2}
	\int_0^t\bigl(\tfrac14L^2\|g(s)\|^2+\|\gtilde(s)\|^2\bigr)\,ds\\
	+\frac{C_Ft^\alpha}{(1-\alpha)^2(2\alpha-1)}\,
	\|u_{0h}\|_1^2\quad\text{for $0<t<T$.}
\end{multline*}
\end{theorem}
\begin{proof}
Using \eqref{eq: spatially discrete}, we find that the 
function~$w_h=u_h-u_{0h}$ satisfies
\[
\bigiprod{w_{ht},v}+\iprod{\partial_t^{1-\alpha}w_{hx},v_x}
	-\bigiprod{F\partial_t^{1-\alpha}w_h,v_x}
	=\iprod{g(t),v}+\iprod{J(t),v_x}
\]
for all $v\in\Sh$, where 
$J(x,t)=\gtilde(x,t)+F(x,t)\partial_t^{1-\alpha}u_{0h}(x) 
-\partial_t^{1-\alpha}(u_{0h})_x(x,t)$.  Choosing 
$v=\partial_t^{1-\alpha}w_h(t)\in\Sh$, 
\begin{multline}\label{eq: energy stab}
\bigiprod{w_{ht},\partial_t^{1-\alpha}w_h}
	+\|\partial_t^{1-\alpha}w_{hx}\|^2
	-\bigiprod{F\partial_t^{1-\alpha}w_h,
		\partial_t^{1-\alpha}w_{hx}}\\
	=\bigiprod{g(t),\partial_t^{1-\alpha}w_h}
	+\bigiprod{J(t),\partial_t^{1-\alpha}w_{hx}},
\end{multline}
and since $\partial_t^{1-\alpha}w_h(x,t)=\partial_t^{1-\alpha}u_h(x,t)
-\omega_\alpha(t)u_{0h}(x)=0$ if $x\in\{0,L\}$, integration by parts 
gives 
\begin{align*}
\bigiprod{F\partial_t^{1-\alpha}w_h,\partial_t^{1-\alpha}w_{hx}}
	=\int_0^L F\,\tfrac12\bigl(
	\bigl(\partial_t^{1-\alpha}w_h\bigr)^2\bigr)_x\,dx
	=-\int_0^LF_x\,\tfrac12\bigl(
		\partial_t^{1-\alpha}w_h\bigr)^2\,dx.
\end{align*}
Hence, by assumption~1,
\begin{equation}\label{eq: wht}
\bigiprod{w_{ht},\partial_t^{1-\alpha}w_h}
	+\|\partial_t^{1-\alpha}w_{hx}\|^2
	\le\bigiprod{g(t),\partial_t^{1-\alpha}w_h}
	+\bigiprod{J(t),\partial_t^{1-\alpha}w_{hx}},
\end{equation}
and the Poincar\'e inequality,
$\|v\|^2\le\tfrac12 L^2\|v_x\|^2$ for $v\in H^1_0(\Omega)$, implies 
that
\[
\bigiprod{g(t),\partial_t^{1-\alpha}w_h}
	\le\frac{L^2}{4}\|g(t)\|^2
	+\frac{1}{2}\|\partial_t^{1-\alpha}w_{hx}\|^2.
\]
Using 
$\iprod{J(t),\partial_t^{1-\alpha}w_{hx}}
\le\tfrac12\|J(t)\|^2+\tfrac12\|\partial_t^{1-\alpha}w_{hx}\|^2$,
and noting that $\partial^{1-\alpha}_tw_h=I^\alpha w_{ht}$ 
because $w_h(0)=0$, it follows from~\eqref{eq: wht} that
\[
\bigiprod{w_{ht},I^\alpha w_{ht}}=
\bigiprod{w_{ht},\partial_t^{1-\alpha}w_h}
	\le\frac{L^2}{4}\|g(t)\|^2+\frac{1}{2}\|J(t)\|^2.
\]
By Lemmas \ref{lem: ||v(t)||^2}~and \ref{lem: I alpha/2}, and 
using the inequality~$\cos(\alpha\pi/2)\ge1-\alpha$, 
\begin{align*}
\|w_h(T)\|^2&\le\frac{T^{1-\alpha}}{(1-\alpha)^2}\int_0^T
	\bigiprod{I^\alpha w_{ht},w_{ht}}\,dt\\
	&\le\frac{T^{1-\alpha}}{(1-\alpha)^2}\biggl(
	\frac{L^2}{4}\int_0^T\|g(t)\|^2\,dt
	+\frac12\int_0^T\|J(t)\|^2\,dt\biggr),
\end{align*}
and since 
$(\partial_t^{1-\alpha}u_{0h})(x,t)=\omega_\alpha(t)u_{0h}(x)$, the 
Cauchy--Schwarz inequality gives
\[
\|J(t)\|\le\|\gtilde\|+\sqrt{C_F}\,\omega_\alpha(t)\|u_{0h}\|_1.
\]
Hence, by assumption~2,
\begin{equation}\label{eq: int ||J||^2}
\frac12\int_0^T\|J(t)\|^2\,dt
	\le\int_0^T\|\gtilde(t)\|^2\,dt
	+C_F\|u_{0h}\|^2_1
		\int_0^T\omega_\alpha(t)^2\,dt,
\end{equation}
and assumption~3 means that 
$\int_0^T \omega_\alpha(t)^2\,dt\le\int_0^T t^{2\alpha-2}\,dt
=T^{2\alpha-1}/(2\alpha-1)$, implying the desired estimate.
\end{proof}

For applications, the condition~$F_x\ge0$ seems unnaturally 
restrictive.  In the next result, we show that it is not necessary 
for stability, but the resulting bound grows more rapidly with~$t$,
owing to the use of the weakly singular Gronwall inequality.

\begin{theorem}\label{thm: stability E}
If we drop assumption~1 from the hypotheses of
Theorem~\ref{thm: stability Fx}, then
\begin{multline*}
\|u_h(t)-u_{0h}\|^2\le
\frac{E_{\alpha/2}\bigl(\tfrac58C_Ft^\alpha/(1-\alpha)\bigr)}%
{(1-\alpha)^2}
	\biggl(t^{1-\alpha}\int_0^t\bigl(\tfrac12\|g(s)\|^2
	+\|\gtilde(s)\|^2\bigr)\,ds\\
	+\frac{C_Ft^\alpha}{2\alpha-1}\,\|u_{0h}\|_1^2
	\biggr)\quad\text{for $0\le t\le T$.}
\end{multline*}
\end{theorem}
\begin{proof}
Recall that $\partial_t^{1-\alpha}w_h=I^\alpha w_{ht}$ because 
$w_h(0)=0$, so \eqref{eq: energy stab} implies that
\[
\bigiprod{I^\alpha w_{ht},w_{ht}}
	\le\tfrac12\|FI^\alpha w_{ht}\|^2
	+\tfrac12\|g(t)\|^2+\tfrac12\|I^\alpha w_{ht}\|^2
	+\tfrac12\|J(t)\|^2.
\]
By Lemma~\ref{lem: I alpha/2},
\[
y_h(T)\equiv\smash[b]{\int_0^T\bigl\|I^{\alpha/2}w_{ht}\bigr\|^2\,dt
	\le\frac{1}{1-\alpha}\int_0^T
	\bigiprod{I^\alpha w_{ht},w_{ht}}\,dt},
\]
so if we let
\[
a(T)=\smash{\frac{1}{2(1-\alpha)}\int_0^T
	\bigl(\|g(t)\|^2+\|J(t)\|^2\bigr)\,dt}
\]
then
\[
y_h(T)\le\smash[t]{a(T)+\frac{C_F}{2(1-\alpha)}\int_0^T
	\|I^\alpha w_{ht}\|^2\,dt}.
\]
Since $I^\alpha w_{ht}=I^{\alpha/2}\bigl(I^{\alpha/2}w_{ht}\bigr)$,
Lemma~\ref{lem: int I alpha} implies that
\[
y_h(T)\le a(T)+b(T)\int_0^T\omega_{\alpha/2}(T-t)y_h(t)\,dt
\quad\text{where}\quad
b(T)=\frac{C_F\omega_{\alpha/2+1}(T)}{2(1-\alpha)}.
\]
Hence, using Lemma~\ref{lem: ||v(t)||^2}~and the Gronwall 
inequality of Lemma~\ref{lem: Gronwall},
\[
\|w_h(t)\|^2\le\frac{t^{1-\alpha}}{1-\alpha}\,y_h(t)
	\le\frac{t^{1-\alpha}}{1-\alpha}\, 
		a(t)E_{\alpha/2}\bigl(b(t)t^{\alpha/2}\bigr),
\]
and the result follows after using \eqref{eq: int ||J||^2} to 
estimate $a(t)$, because the lower bound
$\Gamma(\alpha/2+1)\ge4/5$ for~$1/2<\alpha<1$ implies
$b(t)\le\tfrac58C_Ft^{\alpha/2}/(1-\alpha)$.  Note that 
\eqref{eq: int ||J||^2} does not rely on the first 
assumption~$F_x\ge0$ of Theorem~\ref{thm: stability Fx}.
\end{proof}

To estimate the error in the finite element solution, we will compare 
$u_h(t)$ to the Ritz projection of~$u(t)$. Recall that 
$R_h:H^1_0(\Omega)\to\Sh$ is defined by
\[
\bigiprod{(R_h w)_x,v_x}=\iprod{w_x,v_x}\quad\text{for all $v\in\Sh$,}
\]
and satisfies
\begin{equation}\label{eq: Rh error}
\|v-R_hv\|\le Ch^r|v|_r
\quad\text{and}\quad
\|(v-R_hv)_x\|\le Ch^{r-1}|v|_r
\end{equation}
for $r\in\{1,2\}$ (in our piecewise-linear case).  Here, 
$|v|_r=\|v^{(r)}\|$ is the usual $H^r$-seminorm.
The next theorem shows that if we choose $u_{0h}=R_hu_0$, 
and if $u\in H^1\bigl((0,T),H^r(\Omega)\bigr)$, then 
$\|u_h(t)-u(t)\|=O(h^r)$ for~$0\le t\le T$ and $r\in\{1,2\}$. 

\begin{theorem}\label{thm: uh error}
Let $u_h$ denote the the spatially-discrete finite element solution 
of~\eqref{eq: ibvp}, defined by~\eqref{eq: spatially discrete} with 
$\gtilde(t)\equiv0$.  Then, under the hypotheses of 
Theorem~\ref{thm: stability E}, we have the error bound
\[
\|u_h(t)-u(t)\|^2\le C\|u_{0h}-R_hu_0\|_1^2
	+Ch^{2r}\biggl(\|u_0\|_r^2+\int_0^t\|u_t(s)\|_r^2\,ds\biggr)
\]
for $0\le t\le T$ and $r\in\{1,2\}$, where $C$ depends on $\alpha$, 
$F$, $T$~and $L$.
\end{theorem}
\begin{proof}
We decompose the error into two terms,
\[
u_h-u=\theta+\rho\quad\text{where}\quad
\theta=u_h-R_hu\quad\text{and}\quad
\rho=R_hu-u,
\]
and deduce from~\eqref{eq: spatially discrete} that, for~$v\in\Sh$,
\begin{multline*}
\iprod{\theta_t,v}+\bigiprod{\partial_t^{1-\alpha}\theta_x,v_x}
	-\bigiprod{F\partial_t^{1-\alpha}\theta,v_x}
	=\iprod{g(t),v}\\
	-\iprod{R_hu_t,v}
	-\bigiprod{\partial_t^{1-\alpha}(R_hu)_x,v_x}
	+\bigiprod{F\partial_t^{1-\alpha}R_hu,v_x}.
\end{multline*}
Since $\bigiprod{\partial_t^{1-\alpha}(R_hu)_x,v_x}
=\bigiprod{(R_h\partial_t^{1-\alpha}u)_x,v_x}
=\bigiprod{\partial_t^{1-\alpha}u_x,v_x}$, it follows 
from~\eqref{eq: ibvp weak} that $\theta:[0,T]\to\Sh$ satisfies
\[
\iprod{\theta_t,v}+\bigiprod{\partial_t^{1-\alpha}\theta_x,v_x}
	-\bigiprod{F\partial_t^{1-\alpha}\theta,v_x}
	=\bigiprod{F\partial_t^{1-\alpha}\rho,v_x}-\iprod{\rho_t,v},
\]
which has the same form as~\eqref{eq: spatially discrete}, with 
$\theta$, $-\rho_t$~and $F\partial_t^{1-\alpha}\rho$ playing the 
roles of $u_h$, $g(t)$~and $\gtilde(t)$, respectively.  Hence, 
Theorem~\ref{thm: stability E} gives
\[
\|\theta(T)-\theta(0)\|^2\le 
	C\|\theta(0)\|_1^2
	+C\int_0^T\bigl(\|\rho_t\|^2
		+\|\partial_t^{1-\alpha}\rho\|^2\bigr)\,dt,
\]
and by Lemma~\ref{lem: partial v v_t},
$\int_0^T\|\partial_t^{1-\alpha}\rho\|^2\,dt
	\le C\|\rho(0)\|^2+C\int_0^T\|\rho_t\|^2\,dt$.
The desired error bound follows after applying~\eqref{eq: Rh error} 
with~$v=u_t$.
\end{proof}
\section{An implicit time-stepping scheme}\label{sec: time}
To discretize in time, we suppose $0=t_0<t_1<t_2<\cdots<t_N=T$ and 
denote by~$k_n=t_n-t_{n-1}$ the length of the $n$th 
subinterval~$I_n=(t_{n-1},t_n)$, for $1\le n\le N$.  The maximum time 
step is denoted by~$k=\max_{1\le n\le N}k_n$. With any sequence of 
values 
$v^1$, $v^2$, \dots, $v^N$ we associate the piecewise-constant 
functions $\check v$~and $\hat v$ defined by
\begin{equation}\label{eq: hat check}
\check v(t)=v^n\quad\text{and}\quad
\hat v(t)=v^{n-1}\quad\text{for $t_{n-1}<t<t_n$.}
\end{equation}
Integrating the fractional Fokker--Planck equation~\eqref{eq: ibvp} 
over the $n$th time interval~$I_n$ gives
\begin{equation}\label{eq: integrated FP}
u(t_n)-u(t_{n-1})-\int_{I_n}\partial_t^{1-\alpha}u_{xx}\,dt
	+\int_{I_n}\bigl(F\partial_t^{1-\alpha}u\bigr)_x\,dt
	=\int_{I_n}g(t)\,dt.
\end{equation}
We seek to compute $U^n(x)\approx u(x,t_n)$ for $n=1$, $2$, \dots, 
$N$ by requiring that
\begin{equation}\label{eq: time stepping}
U^n-U^{n-1}-\int_{I_n}\partial_t^{1-\alpha}\check U_{xx}\,dt
	+\int_{I_n}\bigl(F^n\partial_t^{1-\alpha}\check U\bigr)_x\,dt
	=k_n\bar g^n,
\end{equation}
with $F^n(x)=F(x,t_n)$~and 
$\bar g^n\approx k_n^{-1}\int_{I_n}g(t)\,dt$.
The time stepping starts from the initial condition
\begin{equation}\label{eq: U0}
U^0(x)=u_0(x)\quad\text{for $0\le x\le L$,}
\end{equation}
and is subject to the boundary conditions
$U^n(0)=0=U^n(L)$ for $1\le n\le N$.

Since
\[
I^\alpha\check v(t_n)
	=\sum_{j=1}^n\int_{I_j}\omega_\alpha(t_n-s)v^j\,ds
	=\sum_{j=1}^n\omega_{nj}v^j
\]
where
\[
\omega_{nj}=\int_{I_j}\omega_\alpha(t_n-s)\,ds
	=\omega_{1+\alpha}(t_n-t_{j-1})-\omega_{1+\alpha}(t_n-t_j)
	\quad\text{for $n\ge2$,}
\]
with $\omega_{11}=\omega_{1+\alpha}(t_1)$, we see that
\[
\int_{I_n}\partial_t^{1-\alpha}\check v\,dt
	=(I^\alpha\check v)(t_n)-(I^\alpha\check v)(t_{n-1})
	=\sum_{j=1}^n\omega_{nj}v^j-\sum_{j=1}^{n-1}\omega_{n-1,j}v^j.
\]
Hence, to find $U^n$ satisfying~\eqref{eq: time stepping} we 
solve 
\[
U^n-\omega_{nn}U^n_{xx}+\omega_{nn}(F^nU^n)_x=U^{n-1}+k_n\bar g^n
	+\sum_{j=1}^{n-1}(\omega_{nj}-\omega_{n-1,j})
	\bigl(U^j_{xx}-(F^nU^j)_x\bigr).
\]
It follows from Theorem~\ref{thm: time stepping stability} below that 
this linear elliptic boundary-value problem has a unique 
solution~$U^n\in H^1_0(\Omega)$ if $k$ is sufficiently small.
Note that if the mesh is uniform, that is, if $k=k_n$ for all~$n$, 
then the sums are discrete convolutions because
\begin{equation}\label{eq: omega a_n}
\omega_{nj}=\frac{k^\alpha a_{n-j}}{\Gamma(1+\alpha)}
	=\omega_{\alpha+1}(k)a_{n-j}
	\quad\text{where}\quad a_n=(n+1)^\alpha-n^\alpha.
\end{equation}
The next two lemmas, which will help prove a stability estimate 
for~$U^n$, use the following notation for the backward difference,
\[
\partial v(t)=\partial v^n=\frac{v^n-v^{n-1}}{k_n}
\quad\text{for $t\in I_n$.}
\]

\begin{lemma}\label{lem: sum I alpha}
For any sequence~$(v^n)_{n=0}^N$ in~$L_2(\Omega)$,
\[
\sum_{n=1}^N k_n\|(I^\alpha\partial v)(t_n)\|^2
	\le 2\omega_{\alpha+1}(T)\sum_{n=1}^N\omega_{Nn}
	\sum_{j=1}^nk_j\|\partial v^j\|^2
	+2\sum_{n=1}^N k_n^{2\alpha+1}\|\partial v^n\|^2.
\]
\end{lemma}
\begin{proof}
For $t\in I_n$,
\begin{align*}
\|(I^\alpha\partial v)(t_n)\|
	&\le\int_0^t\omega_\alpha(t-s)\|\partial v(s)\|\,ds
	+\int_t^{t_n}\omega_\alpha(t_n-s)\|\partial v(s)\|\,ds\\
	&=(I^\alpha\|\partial v\|)(t)
	+\omega_{\alpha+1}(t_n-t)\|\partial v^n\|,
\end{align*}
where we used the fact that 
$\omega_\alpha(t_n-s)\le\omega_\alpha(t-s)$ because $t\le t_n$.  
Thus, after squaring and integrating over~$t\in I_n$, we obtain
\[
k_n\|(I^\alpha\partial v)(t_n)\|^2
	=\int_{I_n}\|(I^\alpha\partial v)(t_n)\|^2\,dt
	\le 2\int_{I_n}(I^\alpha\|\partial v\|)^2\,dt
	+2k_n^{2\alpha+1}\|\partial v^n\|^2,
\]
since $(2\alpha+1)\Gamma(\alpha+1)^2\ge1$. By 
Lemma~\ref{lem: int I alpha},
\begin{multline*}
\sum_{n=1}^N k_n\|(I^\alpha\partial v)(t_n)\|^2
	\le 2\omega_{\alpha+1}(T)\int_0^T\omega_\alpha(T-t)
	\int_0^t\|\partial v(s)\|^2\,ds\,dt\\
	+2\sum_{n=1}^N k_n^{2\alpha+1}\|\partial v^n\|^2,
\end{multline*}
and the result follows because 
$\int_{I_j}\|\partial v(s)\|^2\,ds=k_j\|\partial v^j\|^2$.
\end{proof}

\begin{lemma}\label{lem: check partial}
For uniform time steps~$k_n=k$ and
for any sequence~$(v_n)_{n=0}^N$, 
\[
\int_{I_n}\partial_t^{1-\alpha}\check v\,dt
	=k(I^\alpha\partial v)(t_n)+\omega_{n1}v^0
\]
and
\[
\sum_{n=1}^N\iprod{v^n,(I^\alpha\check v)(t_n)}
	\ge\tfrac12\omega_{1+\alpha}(k)\sum_{n=1}^N\|v^n\|^2.
\]
\end{lemma}
\begin{proof}
It follows from~\eqref{eq: omega a_n} that 
$\omega_{n-1,j-1}=\omega_{nj}$.  Thus,
\[
I^\alpha\check v(t_{n-1})=\sum_{j=1}^{n-1}\omega_{n-1,j}v^j
	=\sum_{j=2}^n\omega_{n-1,j-1}v^{j-1}
	=\sum_{j=2}^n \omega_{n,j}v^{j-1}
	=I^\alpha\hat v(t_n)-\omega_{n1}v^0
\]
and so
\[
\int_{I_n}\partial_t^{1-\alpha}\check v\,dt
	=I^\alpha\check v(t_n)-I^\alpha\check v(t_{n-1})
	=I^\alpha(\check v-\hat v)(t_n)+\omega_{n1}v^0,
\]
which gives the first result because $\check v-\hat v=k\partial v$.
To prove the second result, use \eqref{eq: omega a_n} to write
$(I^\alpha\check v)(t_n)=\omega_{\alpha+1}(k)\sum_{j=1}^n a_{n-j}v^j$
and apply Lemma~\ref{lem: a_n positive} (from 
Appendix~\ref{sec: positivity}) to deduce that, pointwise in~$x$,
\[
\sum_{n=1}^Nv^n I^\alpha\check v(t_n)=\omega_{\alpha+1}(k)
	\sum_{n=1}^N\sum_{j=1}^na_{n-j}v^nv^j
	\ge\tfrac12\omega_{\alpha+1}(k)\sum_{n=1}^N(v^n)^2.
\]
The desired inequality follows after integrating over~$\Omega$.
\end{proof}

We are now able to show the following stability estimate. 

\begin{theorem}\label{thm: time stepping stability}
Assume $1/2<\alpha\le1$ and consider the implicit 
scheme~\eqref{eq: time stepping} in the case of uniform time 
steps~$k_n=k$. If $U^0\in H^1_0(\Omega)\cap H^2(\Omega)$ and
\[
1+F(x,t_n)^2+F_x(x,t_n)^2\le C_F
	\quad\text{for $0<x<L$ and $1\le n\le N$,}
\]
and if $k$ is sufficiently small, then for $1\le n\le N$, 
\[
\|U^n-U^0\|^2\le t_n
	E_\alpha\biggl(\frac{C_1t_n^{2\alpha}}{\Gamma(\alpha+1)}\biggr)
	\biggl[ C_2\sum_{j=1}^nk\|\bar g^j\|^2
	+C_3\|U^0\|_2^2\biggl(1+ 
	\frac{t_n^{2\alpha-1}}{2\alpha-1}\biggr)\biggr],
\]
where $C_1=24C_F(1+2C_F)$, $C_2=6(1+2C_F)$ and $C_3=6C_F(1+4C_F)$.
\end{theorem}
\begin{proof}
Put $W^n=U^n-U^0$.  Since the mesh is uniform and~$W^0=0$, 
Lemma~\ref{lem: check partial} implies that
\[
\int_{I_n}\partial_t^{1-\alpha}\check U_{xx}\,dt
	=\int_{I_n}\partial_t^{1-\alpha}\check W_{xx}\,dt
	+U^0_{xx}\int_{I_n}\omega_\alpha(t)\,dt
	=k(I^\alpha\partial W_{xx})(t_n)+\omega_{n1}U^0_{xx},
\]
and similarly
\[
\int_{I_n}\bigl(F^n\partial_t^{1-\alpha}\check U\bigr)_x\,dt
	=k\bigl(F^nI^\alpha(\partial U)(t_n)\bigr)_x
	+\omega_{n1}(F^nU^0)_x.
\]
Thus, putting $\Phi^n=U^0_x-F^nU^0$, our time-stepping 
scheme~\eqref{eq: time stepping} implies that
\begin{equation}\label{eq: partial W}
\begin{aligned}
k\,\partial W^n-k(I^\alpha\partial W_{xx})(t_n)
	&=U^n-U^{n-1}-\int_{I_n}\partial_t^{1-\alpha}\check U_{xx}\,dt
	+\omega_{n1}U^0_{xx}\\
	&=k\bar g^n-\int_{I_n}\bigl(F^nI^\alpha\check U\bigr)_x\,dt
	+\omega_{n1}U^0_{xx}\\
	&=k\bar g^n-k\bigl(F^n(I^\alpha\partial W)(t_n)\bigr)_x
	+\omega_{n1}\Phi_x^n.
\end{aligned}
\end{equation}

We take the inner product of both sides of~\eqref{eq: partial W}
with~$(I^\alpha\partial W)(t_n)$, then integrate by parts with respect
to~$x$ and use the fact that $W^n(0)=0=W^n(L)$ to arrive at
\begin{align*}
k\bigiprod{\partial W^n&,(I^\alpha\partial W)(t_n)}
	+k\|(I^\alpha\partial W_x)(t_n)\|^2\\
	&=k\iprod{\bar g^n,(I^\alpha\partial W)(t_n)}
	+\bigiprod{kF^n(I^\alpha\partial W)(t_n)-\omega_{n1}\Phi^n,
	(I^\alpha\partial W_x)(t_n)}\\
	&\le\tfrac12 k\|\bar g^n\|^2
	+\tfrac{1}{2}k\|(I^\alpha\partial W)(t_n)\|^2\\
		&\qquad{}+\tfrac12 k^{-1}\bigl(
		k\|F^n(I^\alpha\partial W)(t_n)\| 
	+\omega_{n1}\|\Phi^n\|\bigr)^2
		+\tfrac{1}{2}k\|(I^\alpha\partial W_x)(t_n)\|^2.
\end{align*}
Since $\|\Phi^n\|^2\le(1+\|F^n\|^2_{L_\infty(\Omega)})\|U^0\|_1^2$,
\begin{multline*}
k\bigiprod{\partial W^n,(I^\alpha\partial W)(t_n)}
	+\tfrac{1}{2}k\|(I^\alpha\partial W_x)(t_n)\|^2
	\le\tfrac12 k\|\bar g^n\|^2\\
	+\bigl(1+\|F^n\|_{L_\infty(\Omega)}^2\bigr)\bigl(
	k\|(I^\alpha\partial W)(t_n)\|^2
	+k^{-1}\omega_{n1}^2\|U^0\|_1^2\bigr),
\end{multline*}
so after summing over~$n$ and applying the second part 
of Lemma~\ref{lem: check partial}, we see that
\begin{multline}\label{eq: next partial W}
\sum_{n=1}^Nk\|(I^\alpha\partial W_x)(t_n)\|^2
	\le\sum_{n=1}^Nk\|\bar g^n\|^2
	+2C_F\sum_{n=1}^Nk\|(I^\alpha\partial W)(t_n)\|^2\\
	+2C_F\|U^0\|_1^2\sum_{n=1}^Nk^{-1}\omega_{n1}^2.
\end{multline}
Recall from~\eqref{eq: omega a_n} that 
$\omega_{n1}=\omega_{\alpha+1}(k)a_{n-1}$ with
$a_n=(n+1)^\alpha-n^\alpha$, and observe that
$a_n\le\alpha n^{\alpha-1}$ for~$n\ge1$.  Thus, for~$k$ sufficiently 
small,
\begin{equation}\label{eq: sum omega_n1^2}
\begin{aligned}
\sum_{n=1}^Nk^{-1}\omega_{n1}^2
	&=\frac{k^{2\alpha-1}}{\Gamma(\alpha+1)^2}\sum_{n=0}^{N-1}a_n^2
	\le\frac{k^{2\alpha-1}}{\Gamma(\alpha+1)^2}\biggl(
	1+(2^\alpha-1)^2+\sum_{n=2}^{N-1}(\alpha n^{\alpha-1})^2\biggr)\\
	&\le 1+\frac{k^{2\alpha-1}}{\Gamma(\alpha)^2}
		\int_1^{N-1}y^{2\alpha-2}\,dy
	\le1+\frac{t_{N-1}^{2\alpha-1}}{2\alpha-1},
\end{aligned}
\end{equation}
where, in the final step, we used the assumption that 
$1/2<\alpha<1$ and the fact that $\Gamma(\alpha)\ge1$.

In a similar fashion, we next take the inner product 
of~\eqref{eq: partial W} with~$\partial W^n$ to obtain
\begin{align*}
k\|\partial W^n\|^2
	&+k\bigiprod{(I^\alpha\partial W_x)(t_n),\partial W_x^n}\\
	&=k\iprod{\bar g^n,\partial W^n}
	-\bigiprod{k\bigl(
		F^n(I^\alpha\partial W)(t_n)\bigr)_x,\partial W^n}
		+\iprod{\omega_{n1}\Phi_x^n,\partial W^n}\\
	&\le\tfrac32k\|\bar g^n\|^2
	+\tfrac32k\bigl\|F^n_x(I^\alpha\partial W)(t_n)
		+F^n(I^\alpha\partial W_x)(t_n)\bigr\|^2\\
	&\qquad{}+\tfrac32k^{-1}\omega_{n1}^2\|\Phi^n_x\|^2
		+(\tfrac16+\tfrac16+\tfrac16)k\|\partial W^n\|^2
\end{align*}
and hence
\begin{multline*}
\tfrac{1}{2}k\|\partial W^n\|^2
	+k\bigiprod{(I^\alpha\partial W_x)(t_n),\partial W^n_x}
	\le\tfrac32 k\|\bar g^n\|^2
	+\tfrac32k^{-1}\omega_{n1}^2\|\Phi^n_x\|^2\\
	+3k\|F^n_x\|_{L_\infty(\Omega)}^2\|(I^\alpha\partial W)(t_n)\|^2
	+3k\|F^n\|_\infty^2\|(I^\alpha\partial W_x)(t_n)\|^2.
\end{multline*}
Since $\|\Phi^n_x\|^2=\|U^0_{xx}-F^n_xU^0-F^nU^0_x\|^2
\le C_F\|U^0\|_2^2$, after summing over~$n$ it 
follows from the second part of Lemma~\ref{lem: check partial} that
\begin{multline*}
Y^N\equiv\sum_{n=1}^Nk\|\partial W^n\|^2
	\le3\sum_{n=1}^Nk\|\bar g^n\|^2
	+3C_F\|U^0\|_2^2
		\sum_{n=1}^Nk^{-1}\omega_{n1}^2\\
	+6C_F\sum_{n=1}^Nk\|(I^\alpha\partial W)(t_n)\|^2
	+6C_F\sum_{n=1}^Nk\|(I^\alpha\partial W_x)(t_n)\|^2,
\end{multline*}
which, together with \eqref{eq: next partial W}~and
\eqref{eq: sum omega_n1^2}, implies that 
\[
Y^N\le\tfrac12A_N+\tfrac14 C_1\sum_{n=1}^Nk
	\|(I^\alpha\partial W)(t_n)\|^2,
\]
where
\[
A_N=C_2\sum_{n=1}^Nk\|\bar g^n\|^2
	+C_3\|U^0\|^2_2
	\biggl(1+\frac{t_N^{2\alpha-1}}{2\alpha-1}\biggr).
\]
Hence, by Lemma~\ref{lem: sum I alpha},
\begin{align*}
Y^N&\le\tfrac12A_N+\tfrac12C_1\biggl(\omega_{\alpha+1}(T)
	\sum_{n=1}^N\omega_{Nn}Y^n+\sum_{n=1}^Nk^{2\alpha+1}
	\|\partial W^n\|^2\biggr)\\
&\le\tfrac12A_N+\tfrac{1}{2}C_1\omega_{\alpha+1}(T)
	\sum_{n=1}^{N-1}\omega_{Nn}Y^n
	+\tfrac12C_1\bigl(\omega_{\alpha+1}(T)\omega_{NN}
	+k^{2\alpha}\bigr)Y^N.
\end{align*}
For $k$ sufficiently small, the term 
in~$Y^N$ on the right-hand side is bounded by~$\tfrac12Y^N$.
Therefore, because 
$\omega_{Nn}\le k^\alpha(N-n)^{\alpha-1}/\Gamma(\alpha)$,
\[
Y^N\le A_N+\frac{B_Nk^\alpha}{\Gamma(\alpha)}\sum_{n=1}^{N-1}
	(N-n)^{\alpha-1}Y^n
	\quad\text{where}\quad
	B_N=C_1\omega_{\alpha+1}(t_N),
\]
and so
\[
Y^n\le A_n+B_Nk\sum_{j=0}^{n-1}\omega_\alpha(t_n-t_j)Y^j
	\quad\text{for $0\le n\le N$.}
\]
Thus, by Lemma~\ref{lem: discrete Gronwall},
$Y^N\le A_NE_\alpha\bigl(B_Nt_N^\alpha\bigr)
=A_NE_\alpha\bigl(C_1t_N^{2\alpha}/\Gamma(\alpha+1)\bigr)$. Finally, 
\[
\|W^n\|^2=\biggl\|\sum_{j=1}^nk\partial W^j\biggr\|^2
	\le\biggl(\sum_{j=1}^nk\biggr)
	\biggl(\sum_{j=1}^nk\|\partial W^j\|^2\biggr)=t_nY^n,
\]
and the result follows.
\end{proof}

We can now prove the following error bound, which implies
\[
\|U^n-u(t_n)\|=O(k^\alpha),
\]
if $u$ is sufficiently regular and if $\|\bar g^j-g(t)\|\le Ck^\alpha$
for~$t\in I_j$; recall that $|v|_r=\|v^{(r)}\|$.

\begin{theorem}\label{thm: time error}
Assume $1/2<\alpha\le1$ and consider the implicit 
scheme~\eqref{eq: time stepping} in the case of uniform time 
steps~$k_n=k$.  If 
$F\in L_\infty\bigl((0,T),W^1_\infty(\Omega)\bigr)$, and if $k$ is 
sufficiently small, then for $0\le t_n\le T$,
\begin{multline*}
\|U^n-u(t_n)\|^2\le C\sum_{j=1}^n\int_{I_j}\|\bar g^j-g(t)\|^2\,dt
	+Ck^{2\alpha-1}\int_0^k t|u_t|_2^2\,dt\\
	+Ck^{2\alpha}\int_k^{t_n}|u_t|_2^2\,dt
	+Ck^2\|u_0\|_1^2+Ck^{2\alpha}\int_0^{t_n}\|u_t\|_1^2\,dt,
\end{multline*}
where $C$ depends on $\alpha$, $F$ and $T$.
\end{theorem}
\begin{proof}
Denote the error at the $n$th time level by~$e^n=U^n-u(t_n)$.
Subtracting \eqref{eq: integrated FP} from~\eqref{eq: time stepping} 
yields 
\[
e^n-e^{n-1}-\int_{I_n}\partial_t^{1-\alpha}\check e_{xx}\,dt
	+\int_{I_n}\bigl(F^n\partial_t^{1-\alpha}\check e\bigr)_x\,dt
	=k\rho^n,
\]
where $\rho^n=\rho^n_1+\rho^n_2+\rho^n_3$ for
\begin{gather*}
\rho_1^n=\bar g^n-\frac{1}{k}\int_{I_n}g(t)\,dt,\qquad
\rho_2^n=\frac{1}{k}\int_{I_n}\partial_t^{1-\alpha}
		(\check u-u)_{xx}\,dt,\\
\rho_3^n=\frac{1}{k}\int_{I_n}\bigl(
	F\partial_t^{1-\alpha}u
	-F^n\partial_t^{1-\alpha}\check u\bigr)_x\,dt.
\end{gather*}
Applying Theorem~\ref{thm: time stepping stability}, with $e^n$~and 
$\rho^n$ playing the roles of $U^n$~and $\bar g^n$, and noting that 
$e^0=0$ by~\eqref{eq: U0}, we see that
\begin{equation}\label{eq: e^N}
\|e^N\|^2\le C\sum_{n=1}^Nk\|\rho^n\|^2
	\quad\text{for $1\le n\le N$.}
\end{equation}
Since $\rho^n_1=k^{-1}\int_{I_n}(\bar g^n-g)\,dt$, we have
\begin{equation}\label{eq: sum rho1}
\sum_{n=1}^N k\|\rho^n_1\|^2
	\le\sum_{n=1}^N\int_{I_n}\|\bar g^n-g\|^2\,dt,
\end{equation}
and if we put
\[
\Lambda_n(s)=\begin{cases}
	\omega_\alpha(t_n-s),&t_{n-1}<s<t_n,\\
	\omega_\alpha(t_n-s)-\omega_\alpha(t_{n-1}-s),&0<s<t_{n-1},
\end{cases}
\]
and 
$\delta_{nj}(t)=(t-t_{j-1})^{-1/2}\int_{t_{j-1}}^t\Lambda_n(s)\,ds$ 
for~$t\in I_j$, then
\begin{align*}
k\rho^n_2&=I^\alpha(\check u-u)_{xx}(t_n)
	-I^\alpha(\check u-u)_{xx}(t_{n-1}) 
	=\int_0^{t_n}\Lambda_n(s)(\check u-u)_{xx}(s)\,ds\\
	&=\sum_{j=1}^n\int_{I_j}\Lambda_n(s)\int_s^{t_j}u_{xxt}(t)\,dt\,ds
	=\sum_{j=1}^n\int_{I_j}\delta_{nj}(t)u_{xxt}(t)
		(t-t_{j-1})^{1/2}\,dt.
\end{align*}
Hence,
\begin{equation}\label{eq: rho2}
\begin{aligned}
\sum_{n=1}^N k\|\rho^n_2\|^2&\le\frac{1}{k}\sum_{n=1}^N\sum_{j=1}^n
	\int_{I_j}\|u_{xxt}(t)\|^2(t-t_{j-1})\,dt
	\int_{I_j}\delta_{nj}(t)^2\,dt\\
	&=\frac{1}{k}\sum_{j=1}^N\int_{I_j}(t-t_{j-1})\|u_{xxt}(t)\|^2\,dt
	\,\sum_{n=j}^N\int_{I_j}\delta_{nj}(t)^2\,dt.
\end{aligned}
\end{equation}
We find that
\[
\delta_{nn}(t)^2\le\int_{t_{n-1}}^t\omega_\alpha(t_n-s)^2\,ds
	=\frac{k^{2\alpha-1}-(t_n-t)^{2\alpha-1}}%
{\Gamma(\alpha)^2(2\alpha-1)}\quad\text{for $t\in I_n$,}
\]
and, since $0<\omega_\alpha(t_n-s)<\omega_\alpha(t_{n-1}-s)$ 
for~$s<t_{n-1}$,
\[
\delta_{n,n-1}(t)^2\le\int_{t_{n-2}}^t\omega_\alpha(t_{n-1}-s)^2\,ds
	=\frac{k^{2\alpha-1}-(t_{n-1}-t)^{2\alpha-1}}%
{\Gamma(\alpha)^2(2\alpha-1)}\quad\text{for $t\in I_{n-1}$,}
\]
whereas if $1\le j\le n-2$, then the Mean Value Theorem implies that
\[
\delta_{nj}(t)^2\le\int_{t_{j-1}}^t
	\bigl[\omega_\alpha'(t_{n-1}-s)k\bigr]^2\,ds
	\le\frac{(1-\alpha)^2}{\Gamma(\alpha)^2}\,(n-1-j)^{2\alpha-4}
	k^{2\alpha-1}\quad\text{for $t\in I_j$,}
\]
so
\[
\int_{I_j}\delta_{nj}(t)^2\,dt\le Ck^{2\alpha}\times
\begin{cases}
	(n-1-j)^{-2\alpha-4},&1\le j\le n-2,\\
	1,&n-1\le j\le n.
\end{cases}
\]
Thus,
\[
\sum_{n=j}^N\int_{I_j}\delta_{nj}(t)^2\,dt
	\le Ck^{2\alpha}\biggl(2+\sum_{n=j+2}^N(n-1-j)^{-2\alpha-4}
	\biggr)\le Ck^{2\alpha},
\]
and therefore by~\eqref{eq: rho2},
\begin{equation}\label{eq: sum rho2}
\sum_{n=1}^N k\|\rho^n_2\|^2\le Ck^{2\alpha-1}\sum_{n=1}^N
	\int_{I_n}(t-t_{n-1})\|u_{xxt}\|^2\,dt.
\end{equation}

It remains to deal 
with~$\rho^n_3=\rho^n_{31}+\rho^n_{32}$, where
\begin{align*}
\rho^n_{31}&=\frac{1}{k}\int_{I_n}\Bigl(
	F_x^n\partial_t^{1-\alpha}(u-\check u)
	+F^n\partial_t^{1-\alpha}(u-\check u)_x\Bigr)\,dt,\\
\rho^n_{32}&=\frac{1}{k}\int_{I_n}\Bigl(
	(F-F^n)_x\partial_t^{1-\alpha}u
	+(F-F^n)\partial_t^{1-\alpha}u_x\Bigr)\,dt.
\end{align*}
Estimating $\rho^n_{31}$ in the same way as~$\rho^n_2$, we see that
\begin{equation}\label{eq: sum rho31}
\sum_{n=1}^Nk\|\rho^n_{31}\|^2 
	\le Ck^{2\alpha-1}\sum_{n=1}^N 
	\int_{I_n}(t-t_{n-1})\|u_t\|_1^2\,dt
	\le Ck^{2\alpha}\int_0^{t_N}\|u_t\|_1^2\,dt.
\end{equation}
Next, since $\|F(t)-F^n\|_1\le Ck$ for~$t\in I_n$,
\[
\|\rho^n_{32}\|^2
	\le k^{-2}\int_{I_n}\bigl\|F(t)-F^n\bigr\|_1^2\,dt
	\int_{I_n}\|\partial_t^{1-\alpha}u\|_1^2\,dt
	\le Ck\int_{I_n}\|\partial_t^{1-\alpha}u\|_1^2\,dt,
\]
so, using Lemma~\ref{lem: partial v v_t},
\begin{equation}\label{eq: sum rho32}
\sum_{n=1}^Nk\|\rho^n_{32}\|^2\le Ck^2\int_0^{t_N}
	\|\partial_t^{1-\alpha}u\|_1^2\,dt
	\le Ck^2\biggl(\|u_0\|_1^2+\int_0^{t_N}\|u_t\|_1^2\,dt\biggr).
\end{equation}
The error bound now follows from \eqref{eq: e^N}, 
\eqref{eq: sum rho1} and \eqref{eq: sum rho2}--\eqref{eq: sum rho32}.
\end{proof}
\section{Numerical experiments}\label{sec: numerical}
Our discrete-time solution~$U^n\in H^1_0(\Omega)$ 
of~\eqref{eq: time stepping} satisfies
\[
\iprod{U^n-U^{n-1},v}
	+\int_{I_n}\bigiprod{\partial_t^{1-\alpha}\check U_x,v_x}\,dt
	-\int_{I_n}\bigiprod{F^n\partial_t^{1-\alpha}\check U,v_x}\,dt
	=\int_{I_n}\iprod{g,v}\,dt
\]
for all~$v\in H^1_0(\Omega)$.  We therefore seek a 
fully-discrete solution~$U^n_h\in\Sh$ given by
\[
\iprod{U^n_h-U^{n-1}_h,v}
	+\int_{I_n}\bigiprod{\partial_t^{1-\alpha}\check U_{hx},v_x}\,dt
	-\int_{I_n}\bigiprod{F^n\partial_t^{1-\alpha}\check U_h,v_x}\,dt
	=\int_{I_n}\iprod{g,v}\,dt
\]
for all $v\in\Sh$ and for $1\le n\le N$, with $U^0_h=R_hu_0$.  
(In our case, the Ritz projection~$R_hu_0$ is simply the nodal 
interpolant to~$u_0$.) Explicitly, let $\phi_p\in\Sh$ denote the $p$th 
nodal basis function, so that $\phi_p(x_q)=\delta_{pq}$ and
\[
U^n_h(x)=\sum_{p=1}^{P-1}U^n_p\phi_p(x) \quad
	\text{where $U^n_p=U^n_h(x_p)\approx U^n(x_p)\approx u(x_p,t_n)$.}
\]
Define the $(P-1)\times(P-1)$ tridiagonal matrices $\matM$~and 
$\matB^n$ with entries
$M_{pq}=\iprod{\phi_q,\phi_p}$ and
$B^n_{pq}=\iprod{\phi_{qx},\phi_{px}}-\iprod{F^n\phi_q,\phi_{px}}$,
and define $(P-1)$-dimensional 
column vectors $\vecU^n$~and $\vecG^n$ with components $U^n_p$~and 
$G^n_p=\int_{I_n}\iprod{g,\phi_p}\,dt$.  We find that
\[
\matM\vecU^n-\matM\vecU^{n-1}+\sum_{j=1}^n\omega_{nj}\matB^n\vecU^j
	-\sum_{j=1}^{n-1}\omega_{n-1,j}\matB^n\vecU^j=\vecG^n,
\]
so at the $n$th time step we must solve the linear system
\[
\bigl(\matM+\omega_{nn}\matB^n\bigr)\vecU^n=\matM\vecU^{n-1}+\vecG^n
	-\sum_{j=1}^{n-1}\bigl(\omega_{nj}-\omega_{n-1,j}\bigr)\matB^n
	\vecU^j.
\]
We now describe some experiments using this numerical scheme.  
\subsection{Convergence behaviour}

\begin{table}
\caption{Behaviour of~$E_{N,h}$, defined by~\eqref{eq: ENh}, as the 
number of time steps~$N$ increases, for different choices of the 
mesh grading parameter~$\gamma$. In each case, $\alpha=0.625$ and 
we use a spatial resolution~$h=L/P$ with~$P=5120$.}
\label{tab: max norm error}
\begin{center}
\begin{tabular}{r|cc|cc|cc}
\hline
$N$&
$\gamma=1.0$&$r_t$&
$\gamma=\alpha^{-1}=1.6$&$r_t$&
$\gamma=2.0$&$r_t$\\
\hline
 80  & 8.93e-03 &       & 8.60e-03 &       & 1.01e-02 \\ 
160  & 4.95e-03 & 0.851 & 4.33e-03 & 0.989 & 5.09e-03 & 0.986\\ 
320  & 2.80e-03 & 0.823 & 2.18e-03 & 0.993 & 2.56e-03 & 0.992\\ 
640  & 1.62e-03 & 0.791 & 1.09e-03 & 0.996 & 1.28e-03 & 0.995\\
\hline
\end{tabular} 
\end{center}
\end{table}

\begin{table}
\caption{Behaviour of~$E_{N,h}$, defined by~\eqref{eq: ENh}, as the 
number of time steps~$N$ increases, for different choices 
of~$\alpha$.  In each case, $\gamma=1$ and we use a spatial 
resolution~$h=L/P$ with~$P=5120$.}
\label{tab: fixed h var alpha}
\begin{center}
\begin{tabular}{r|cc|cc|cc}
\hline
$N$&
$\alpha=0.25$&$r_t$&
$\alpha=0.50$&$r_t$&
$\alpha=0.75$&$r_t$\\
\hline
 80  & 1.93e-01 &       & 2.21e-02 &       & 7.40e-03 \\ 
160  & 1.70e-01 & 0.183 & 1.50e-02 & 0.554 & 3.73e-03 & 0.989\\ 
320  & 1.50e-01 & 0.188 & 1.04e-02 & 0.538 & 1.88e-03 & 0.990\\ 
640  & 1.31e-01 & 0.193 & 7.20e-03 & 0.525 & 9.46e-04 & 0.989\\ 
\hline
\end{tabular} 
\end{center}
\end{table}

\begin{table}
\caption{Behaviour of~$E_{N,h}$, defined by~\eqref{eq: ENh}, as the 
spatial resolution~$h=L/P$ increases, for different choices 
of~$\alpha$.  In each case, $\gamma=\alpha^{-1}$ and
we use $N=10,000$ time steps.}
\label{tab: fixed N var alpha}
\begin{center}
\begin{tabular}{r|cc|cc|cc}
\hline
$P$&
$\alpha=0.25$&$r_x$&
$\alpha=0.50$&$r_x$&
$\alpha=0.75$&$r_x$\\
\hline
 4  & 8.43e-02 &       & 8.22e-02 &       & 7.74e-02 \\ 
 8  & 2.97e-02 & 1.505 & 2.92e-02 & 1.495 & 2.77e-02 & 1.483\\ 
16  & 6.21e-03 & 2.258 & 6.07e-03 & 2.264 & 5.75e-03 & 2.268\\ 
32  & 1.50e-03 & 2.052 & 1.46e-03 & 2.054 & 1.39e-03 & 2.046\\ 
64  & 3.47e-04 & 2.108 & 3.23e-04 & 2.177 & 3.03e-04 & 2.201\\ 
\hline
\end{tabular} 
\end{center}
\end{table}

\begin{figure}
\caption{Estimated convergence rate $r_t$ as a function of~$\alpha$, 
with uniform time steps.}
\label{fig: conv rate}
\begin{center}
\includegraphics[scale=0.4]{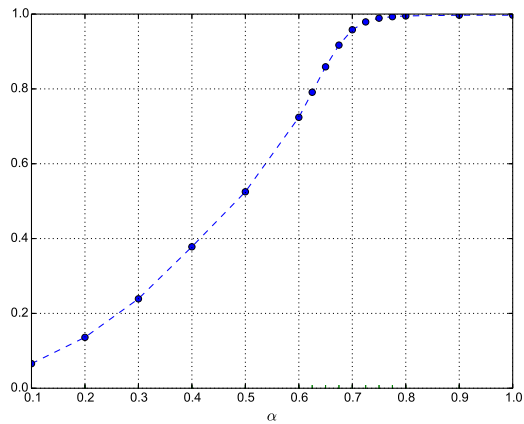}
\end{center}
\end{figure}

\begin{figure}
\caption{Contour plot of the solution for the problem of 
Section~\ref{sec: application}.  The dashed line shows the first 
moment~$\bar x(t)$.}
\label{fig: application}
\begin{center}
\includegraphics[scale=0.4]{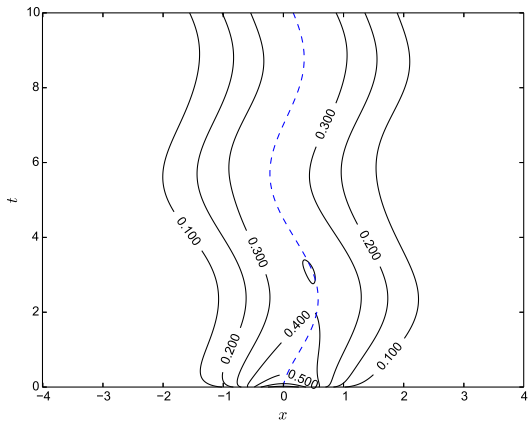}
\end{center}
\end{figure}

\begin{figure}
\caption{Top: first moment (as computed via Laplace transforms) of 
the solution for the problem of 
Section~\ref{sec: application}.  Bottom: error in the first moment 
of~$U^n_h$.}
\label{fig: x bar}
\begin{center}
\includegraphics[scale=0.4]{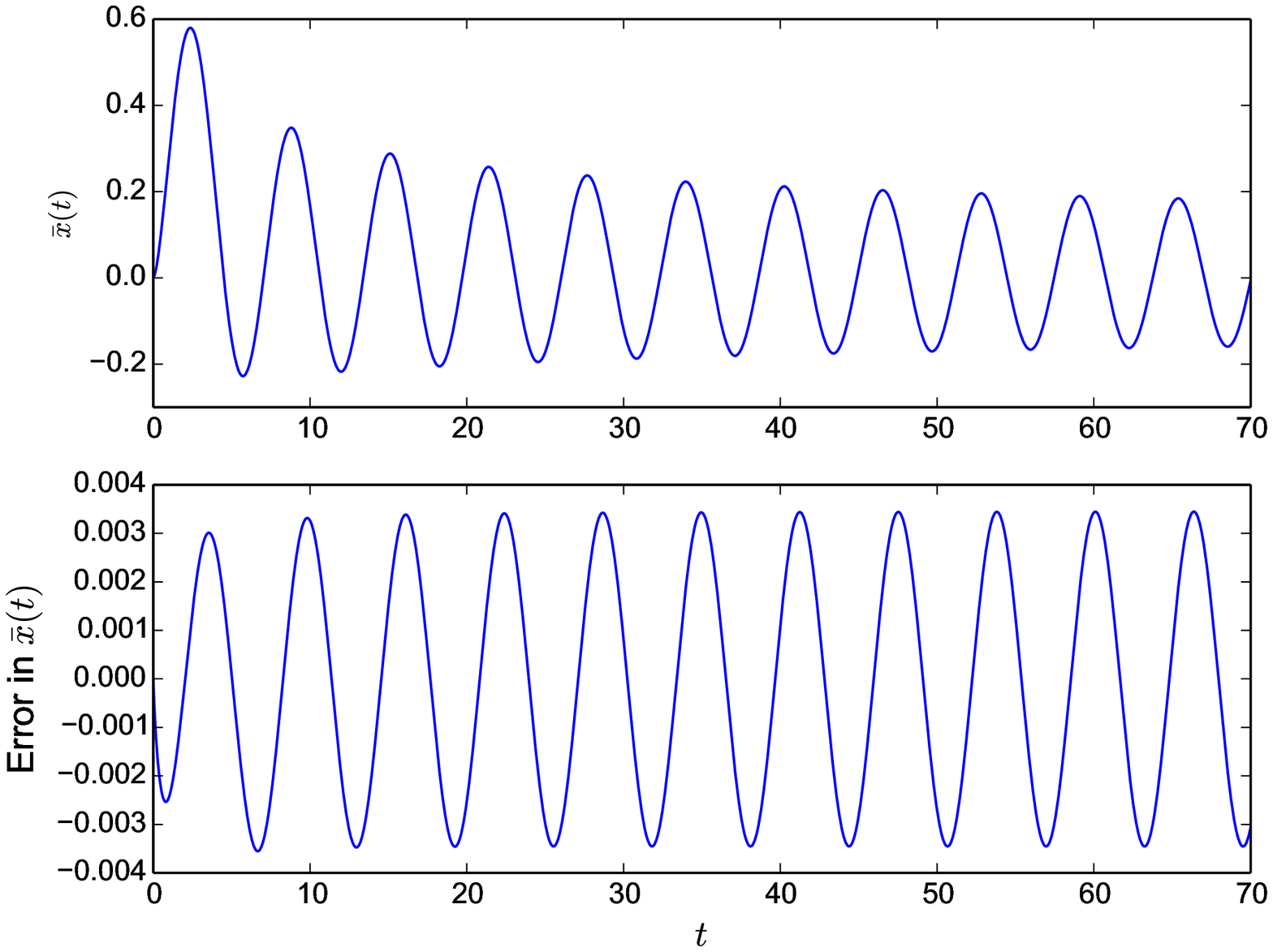}
\end{center}
\end{figure}

In our first test problem, we considered~\eqref{eq: ibvp} with 
\[
F(x,t)=x+\sin t,\quad T=1,\quad L=\pi,\quad\kappa_\alpha=\mu_\alpha=1,
\]
where the source term~$g$ was chosen so that 
$u(x,t)=[1+\omega_{1+\alpha}(t)]\sin x$.  It follows that 
$u_t=O(t^{\alpha-1})$ as~$t\to0^+$, and this singular behaviour is 
known to be typical~\cite{McLean2010} for the fractional diffusion 
equation (that is, when the lower-order term in~$F$ is absent).  We
employed a uniform spatial grid with~$h=\pi/P$, but allowed a 
nonuniform spacing in time by putting
\[
t_n=(n/N)^\gamma T,\quad\text{where $\gamma\ge1$.}
\]
Thus, $\gamma=1$ gives a uniform mesh with~$k=T/N$, but if $\gamma>1$
then the time step is initially~$k_1=T/N^\gamma=O(k^\gamma)$ and 
increases monotonically up to a maximum 
of~$k=k_N\approx\gamma T/N$. Such meshes~\cite{McLeanMustapha2009} 
are commonly used to compensate for singular behaviour 
in the derivatives of~$u$ at~$t=0$. As a measure of the error in the 
numerical solution, we computed
\begin{equation}\label{eq: ENh}
E_{N,h}=\max_{0\le n\le N}\|U^n_h-u(t_n)\|_{L_2(\Omega)}, 
\end{equation}
(where the spatial $L_2$-norm was evaluated via Gauss quadrature) and 
sought to estimate the convergence rates $r_t$~and $r_x$ such that
\[
E_{N,h}\approx C_1k^{r_t}+C_2h^{r_x},
\]
from the relations
\[
\begin{aligned}
r_t&\approx r_t(N,h)=\log_2(E_{N,h}/E_{2N,h})&
	&\text{when $h^{r_x}\ll k^{r_t}$,}\\
r_x&\approx r_x(N,h)=\log_2(E_{N,2h}/E_{N,h})&
	&\text{when $k^{r_t}\ll h^{r_x}$.}
\end{aligned}
\]

We first tested the convergence behaviour with respect to the time 
discretization.  Table~\ref{tab: max norm error} shows how~$E_{N,h}$ 
varies with~$N$, for a fixed, high-resolution spatial grid with 
$P=5120$~subintervals, when $\alpha=0.625$ and for three choices 
of~$\gamma$. In the case of a uniform mesh ($\gamma=1$), we observe 
$r_t\approx0.8$, suggesting that the $O(k^\alpha)$ error bound of 
Theorem~\ref{thm: time error} is somewhat pessimistic in this case.  
Although the convergence analysis of our time-stepping scheme applies 
only when~$\gamma=1$, we observe that $E_{N,h}\approx Ck$
if $\gamma\ge\alpha^{-1}=1.6$. (The constant~$C$ is smallest 
when~$\gamma=\alpha^{-1}$.)  Table~\ref{tab: fixed h var alpha} shows 
results for three different choices of~$\alpha$ as we vary $N$, using 
uniform time steps ($\gamma=1$) and the same fixed spatial grid as 
before.  Note that the choices $\alpha=0.25$~and $\alpha=0.5$ are 
excluded by our theory, which requires $1/2<\alpha<1$. 
Figure~\ref{fig: conv rate} gives a more complete picture of 
the convergence rate~$r_t$ as a function of~$\alpha$ when~$\gamma=1$, 
and may be compared with the known result $r_t=\min(2\alpha,1)$ for 
the homogeneous  diffusion equation (that is, the special case 
$F=0$~and $g=0$) with regular initial 
data~\cite[Lemma~6]{McLeanMustapha2015}.

Next, we tested how $E_{N,h}$ behaves as the spatial mesh is refined,
using a fixed, high-resolution time discretization with~$N=10,000$.
Table~\ref{tab: fixed N var alpha} shows results for three different 
choices of~$\alpha$ using a mesh grading~$\gamma=\alpha^{-1}$ in each 
case.  We see that $E(N,h)\approx C_1h^2$, consistent with 
Theorem~\ref{thm: uh error} (when~$1/2<\alpha<1$).

\subsection{An application}\label{sec: application}
In our second example, we solve the homogeneous equation on the 
spatial interval~$(-L,L)$, that is,
\[
u_t-\partial_t^{1-\alpha}u_{xx}
	+\bigl(F\partial_t^{1-\alpha}u\bigr)_x
	=0\quad\text{for $0<t<T$ and $-L<x<L$,}
\]
with $F=-x+\sin t$, subject to the boundary
conditions $u(\pm L,t)=0$.  For the initial data~$u_0$, 
we chose a normal probability density function with mean~$0$ and
variance~$\sigma^2$.
This choice of~$F$ is taken from a recent paper by Angstmann et 
al.~\cite{AngstmannEtAl2015}; notice that $F_x=-1<0$ so the first 
assumption of Theorem~\ref{thm: stability Fx} is not satisfied and 
we must rely on Theorem~\ref{thm: stability E} to ensure stability
of the spatially discrete scheme~\eqref{eq: spatially discrete}.
For our computations, we used the values $\alpha=0.75$~and
$\sigma=0.5$, with a mesh grading parameter~$\gamma=1/\alpha$. 
Figure~\ref{fig: application} shows a contour plot of the numerical 
solution computed using our fully discrete method in the case $L=9$, 
$T=10$, $N=100$~and $P=2L^2$. Although we do not know an analytical 
solution, Laplace transform techniques~\cite{AngstmannEtAl2015} show 
that in the limiting case when~$L\to\infty$, and 
interpreting~$u(\cdot,t)$ as a probability density function, the 
expected position, or first moment, is
\[
\bar x(t)=\int_{-\infty}^\infty xu(x,t)\,dx
	=\frac{1}{\Gamma(\alpha)}\int_0^tE_\alpha\bigl(-(t-s)\bigr)
	s^{\alpha-1}\sin s\,ds,
\]
where $E_\alpha$ denotes the Mittag--Leffler 
function~\eqref{eq: ML func}.  Figure~\ref{fig: x bar} shows the 
oscillatory behaviour of~$\bar x(t)$ for~$0\le t\le T=70$, and 
the difference between this theoretical value and the first moment of 
the numerical solution~$U^n_h$, in the case $L=20$, 
$N=20T$~and $P=2TL^2$.  We observe little if any loss of accuracy 
over more than 10~oscillations.

\subsection{Non-smooth initial data}
In the special case of a fractional diffusion equation 
($F\equiv0$~and $g\equiv0$), a standard energy argument shows 
that both the exact solution and the finite 
element solution are stable in~$L_2(\Omega)$, with
\[
\|u(t)\|\le\|u_0\|\quad\text{and}\quad
\|u_h(t)\|\le\|u_{0h}\|\quad\text{for $t>0$.}
\]
By comparison, for nonzero~$F$ the stability estimates of 
Theorems \ref{thm: stability Fx}~and \ref{thm: stability E} yield 
weaker bounds of the form
\begin{equation}\label{eq: u_0 stab 1}
\|u_h(t)\|\le C\|u_{0h}\|_1\quad\text{for $0\le t\le T$,}
\end{equation}
and in the case of our (spatially continuous) time-stepping scheme, 
Theorem~\ref{thm: time stepping stability},
\begin{equation}\label{eq: u_0 stab 2}
\|U^n\|\le C\|U^0\|_2\quad\text{for $0\le t\le T$.}
\end{equation}
To investigate whether the stability properties really depend on 
the smoothness of the initial data, we solved the same problem as 
in Section~\ref{sec: application} but chose the nodal values for 
the discrete initial data~$u_{0h}$ to be uniformly distributed 
pseudorandom numbers in the unit interval.  For $0\le t\le T=40$
and many different combinations of $N$~and $P$, we never observed any 
kind of instability.  In all cases, the solution quickly smoothed and 
began an oscillatory behaviour similar to that seen in 
Figure~\ref{fig: application}, suggesting that 
\eqref{eq: u_0 stab 1}~and \eqref{eq: u_0 stab 2} are pessimistic with 
respect to the regularity required of the initial data.
\appendix
\section{Positivity of discrete convolution operators}
\label{sec: positivity}

Recall the following positivity property of Fourier cosine series.

\begin{lemma}\label{lem: Zygmund}
If the sequence $a_0$, $a_1$, $a_2$, \dots tends to zero and satisfies
\[
a_n\ge0\quad\text{and}\quad
a_{n+1}\le\tfrac12(a_n+a_{n+2})\quad\text{for all $n\ge0$,}
\]
then
\[
\frac{a_0}{2}+\sum_{n=1}^\infty a_n\cos n\theta\ge0
        \quad\text{for $-\pi\le\theta\le\pi$.}
\]
\end{lemma}
\begin{proof}
Zygmund~\cite[p.~93 and Theorem~(1.5), p.~183]{Zygmund}.
\end{proof}

For $0<\alpha<1$, let
\[
(AU)^n=\sum_{j=0}^n a_{n-j}U^j
        \quad\text{where $a_n=(n+1)^\alpha-n^\alpha$.}
\]
We used the following inequality in the proof of
Lemma~\ref{lem: check partial}.

\begin{lemma}\label{lem: a_n positive}
For any real, square-summable sequence $U^0$, $U^1$, $U^2$, \dots,
\[ 
\sum_{n=0}^\infty (AU)^nU^n\ge\frac{1}{2}\sum_{n=0}^\infty(U^n)^2.
\]
\end{lemma}
\emph{Proof}.
Define $\Vtilde(\theta)=\sum_{n=0}^\infty V^ne^{in\theta}$, and 
observe that
\[
\int_{-\pi}^\pi\Utilde(\theta)\overline{\Vtilde(\theta)}\,d\theta
        =\sum_{n=0}^\infty\sum_{j=0}^\infty U^n\overline{V^j}
        \int_{-\pi}^\pi e^{i(n-j)\theta}\,d\theta
        =2\pi\sum_{n=0}^\infty U^n\overline{V^n}.
\]
Since
\[
\widetilde{AU}(\theta)
        =\sum_{n=0}^\infty\biggl(\sum_{j=0}^n a_{n-j}U^j\biggr)
                e^{in\theta}
        =\sum_{j=0}^\infty\biggl(\sum_{n=j}^\infty 
                a_{n-j}e^{i(n-j)\theta}\biggr)U^je^{ij\theta}
        =\atilde(\theta)\Utilde(\theta)
\]
we conclude
\[
\sum_{n=0}^\infty (AU)^n\overline{V^n}=\frac{1}{2\pi}\int_{-\pi}^\pi
        \atilde(\theta)\Utilde(\theta)
                \overline{\Vtilde(\theta)}\,d\theta.
\]
In particular, when $V^n=U^n$ is purely real,
\[
\sum_{n=0}^\infty (AU)^nU^n=\frac{1}{2\pi}\int_{-\pi}^\pi
        \Re\atilde(\theta)|\Utilde(\theta)|^2\,d\theta.
\]
The function $f(x)=(x+1)^\alpha-x^\alpha$ is positive for~$x\ge0$,
and as~$x\to\infty$,
\[
f(x)=x^\alpha[(1+x^{-1})^\alpha-1]=x^\alpha[\alpha x^{-1}+O(x^{-2})]
        =\alpha x^{\alpha-1}+O(x^{\alpha-2}),
\]
so in particular $f(x)\to0$.
Furthermore, $f$ is convex because
\[
f''(x)=\alpha(\alpha-1)\bigl[(x+1)^{\alpha-2}-x^{\alpha-2}\bigr]
        =\alpha(1-\alpha)\bigl[x^{\alpha-2}-(x+1)^{\alpha-2}\bigr]\ge0
\]
for all $x>0$, so the sequence $a_n=f(n)$ satisfies the assumptions
of Lemma~\ref{lem: Zygmund}.  Hence,
$\Re\atilde(\theta)\ge a_0/2=1/2$ and finally
\[
\sum_{n=0}^\infty (AU)^nU^n\ge\frac{1}{2\pi}\int_{-\pi}^\pi
        \frac{1}{2}|\Utilde(\theta)|^2\,d\theta
        =\frac{1}{2}\sum_{n=0}^\infty(U^n)^2.\qquad\endproof
\]

\bibliographystyle{siam}
\bibliography{Le_McLean_Mustapha_refs}
\end{document}